\documentclass{article}
\usepackage{graphicx} % Required for inserting images
\usepackage[utf8]{inputenc}
\usepackage{enumitem}
\usepackage{varwidth}
\usepackage{soul}
\usepackage{cancel}
\usepackage{tasks}
\usepackage[italicdiff]{physics}
\usepackage{amsmath}
\usepackage{amssymb}
\usepackage{algpseudocode}
\usepackage{enumitem}
\usepackage{subcaption}
\usepackage[top=2.5cm, bottom=2.5cm, left=4cm, right=4cm]{geometry}
\usepackage{graphicx}
\usepackage{amsthm} % for theorem environments
\newtheorem{theorem}{Theorem}
\newtheorem{prop}[theorem]{Proposition}
\newtheorem{corollary}[theorem]{Corollary}
\newtheorem{remark}[theorem]{Remark}

\newenvironment{notation}
{\par\noindent\textbf{Notations:}\quad}  
  {\par}   
\newtheorem{assumption}{Assumption}

\usepackage{graphicx}

\usepackage[colorinlistoftodos]{todonotes}

\newtheorem{lemma}[theorem]{Lemma}

\providecommand{\keywords}[1]
{
  \small	
  \textbf{\textit{Keywords---}} #1
}
\providecommand{\classification}[1]
{
  \small	
  \textbf{{AMS subject classifications--}} #1
}

\title{Linearized Localized Orthogonal Decomposition for Quasilinear Nonmonotone Elliptic PDE}\author{Maher Khrais\footnotemark[1] \and
 Barbara Verfürth\footnotemark[1]}  
\date{}

\begin{document}

\maketitle
\renewcommand{\thefootnote}{\fnsymbol{footnote}}
\footnotetext[1]{Institut für Numerische Simulation, Universität Bonn, Friedrich-Hirzebruch-Allee 7, D-53115 Bonn, Germany}
\renewcommand{\thefootnote}{\arabic{footnote}}

\begin{abstract}
     In this paper, we propose and analyze a multiscale method for a class of quasilinear elliptic problems of nonmonotone type with spatially multiscale coefficient. The numerical approach is inspired by the Localized Orthogonal Decomposition (LOD), so that we do not require structural assumptions such as periodicity or scale separation and only need minimal regularity assumptions on the coefficient.
     To construct the multiscale space, we solve linear fine-scale problems on small local subdomains, for which we consider two different linearization techniques. For both, we present a rigorous well-posedness analysis and convergence estimates in the $H^1$-semi norm. We compare and discuss theoretically and numerically the performance of our strategies for different linearization points. Both numerical experiments and  theoretical analysis demonstrate the validity and applicability of the method. 

\end{abstract}

 \\

\classification{65N30, 65N15, 35J60, 65J15, 65N12 }

\section{Introduction}\label{SectionIntro}
Many applications lead to partial differential equations with \emph{nonlinearities} and spatially \emph{multiscale} coefficients. This includes stationary Richards equation in the context of groundwater flow~\cite{RicharEq} or heat conductivity of modern composite materials~\cite{heatConduction}. Nonlinear models are considered to accurately model the physical responses of these materials, in particular at high temperatures, intensities or forces, where standard linear approximations are no longer valid. In addition, spatial multiscale coefficients model the heterogeneity of the media. This occurs in engineered composites as well as in nature. For instance, the conductivity in soil formations may change drastically over small distances.

In this work, we focus on a class of nonmonotone quasilinear elliptic PDEs of the form
\begin{equation}\label{NonLinear}
       \begin{split}
           -\nabla\cdot{(\alpha(x,u)\nabla{u})}&=f \quad\text{in} \;\Omega 
        \end{split}
\end{equation}
in an open bounded domain $\Omega \subset \mathbb{R}^d$, where $d\leq 3$. Boundary conditions and detailed assumptions on $\alpha$ will be made precise below.  
In general, the standard numerical approaches fail to describe in a satisfactory way the macro-scale behavior of the solution in the presence of high heterogeneity. In such approaches, the fine spatial features characterized by the parameter $\varepsilon$ are difficult to resolve in the nonasymptotic regime, i.e., when the mesh size $H$ satisfies $H> \varepsilon$.  To recover the optimal convergence rate w.r.t. $H^1$-norm, $H$ must satisfies $H\ll\varepsilon$. However, such a numerical simulation would be prohibitively expensive. 

 The macroscopic behavior of the solution can be rigorously described by mathematical homogenization theory. In the limit $\varepsilon \xrightarrow{}0$, one may replace the original multiscale problem with the homogenized problem, whose coefficients are only slowly varying with respect to space. Mathematical homogenization of~\eqref{NonLinear} has been studied in several papers, for example~\cite{MathematicalHomogenizaiton}. 
Generally speaking, numerical homogenization methods are mainly based on a macroscopic solver where the missing macroscopic data (similar to the homogenized tensor) or suitable non-polynomial multiscale basis functions are obtained through microscopic solvers on local subdomains. 
Several numerical homogenization techniques have been proposed in the past few decades, especially for linear elliptic problems, see~\cite{HMM, GMSFEM, MsFEMLINEAR, LOD-Linear}, but the literature for nonlinear problems, in particular of nonmonotone type~\eqref{NonLinear}, is much more sparse.

The heterogeneous finite element method has been formulated and analyzed for the nonmonotone nonlinear elliptic problem~\eqref{NonLinear} in~\cite{HMM_quasiliner, HMS-ofNonmonotne}. The numerical approach based on the multiscale finite element method MsFEM is studied for~\eqref{NonLinear} with monotonicity assumption imposed on the coefficient in order to derive a convergence rate in~\cite{MSFEM}. The generalized multiscale FEM (GMsFEM) technique is implemented and analyzed to solve nonlinear elliptic problems in~\cite{NonlinearGMS}. In the case of non-stationary Richard equations, the linearized system is tackled by utilizing constraint energy minimizing GMsFEM (CEM-GMsFEM) to obtain proper offline multiscale basis functions, see~\cite{CEM-GMsFEM}. Relying on the partition of unity method to devise  global approximation spaces from a local reduced space, a component-based parametric model order reduction (CB-pMOR) technique is proposed for the parametrized nonmonotone version of problem~\eqref{NonLinear} in~\cite{CB_pMOR}.

In this work, we consider a numerical homogenization approach inspired by the Localized Orthogonal Decomposition method (LOD), which has already successfully been applied to a range of (linear) problems, see~\cite{NumericalHomogenizationActaNumerica,LOD-Linear} for an overview. The goal of the LOD method is to incorporate the (spatial) fine-scale behavior of the coefficient $\alpha$ into the basis of the coarse FE space resulting in a new modified low-dimensional function space with good approximation properties. The fine-scale incorporation into the coarse FE space is obtained by using correction operators. This operator is computed in a localized fashion on small patches of coarse mesh elements. However, these arguments cannot easily be transferred to the nonlinear case. For semilinear equations, the nonlinearity can be neglected in the construction of the correction operators because it is in a lower order term, see~\cite{semilinear}. For quasilinear problems of monotone type, a multiscale method was proposed and analyzed in~\cite{Barbara} which uses a linearization of the problem to compute the correction operators. The idea of linearization was later combined with an update strategy in~\cite{MAir&barbara} for nonlinear Helmholtz equations. Such an iterative numerical homogenization method, but based on generalized polyharmonic splines, is studied in~\cite{IterativeLOD} for a certain class of quasilinear monotone PDE. 

In this work, we consider the original idea from~\cite{Barbara} and once construct and compute a multiscale approximation space in a linearized fashion. This space is employed for the nonlinear multiscale problem with a Galerkin approach, which in practice of course uses an iterative method. We focus on the construction of the multiscale space and, in particular, on the linearization techniques involved for our nonmonotone PDE~\eqref{NonLinear}.
This nonmonotone nature requires several changes in the design of the method and, in particular, its error analysis. Our main contribution is thus to extend the ideas from~\cite{Barbara} to~\eqref{NonLinear}, which is a non-trivial task.
Our a priori error estimate is of optimal order with respect to the mesh size without any dependence on the spatial variations of $\alpha$ or (higher) regularity of the exact solution. Essentially, the error splits into a discretization error similar to the linear case and a linearization error, which we then further analyze. We further discuss the choice of the linearization point for the correction problems.

The paper is structured as follows. In the rest of Section~\ref{SectionIntro}, we introduce the model problem and summarize some essential findings on the existence and uniqueness of the solution and its finite element approximation. In Section~\ref{SectionMultiscaleMethod}, we introduce our multiscale approach and the linearization techniques that are needed for the correction computation.  We  provide a detailed presentation of several numerical experiments  in Section~\ref{SectionExperiment} in which we apply our multiscale techniques with two linearization methods. These experiments aim to demonstrate the effectiveness and applicability of the methods developed in the previous section.  Section~\ref{SectionError} is dedicated to a detailed and rigorous error analysis of our numerical method. The error analysis interprets the practical observations within a theoretical framework, and discusses the impact of the linearization points noted in the preceding section. We also discuss the linearization error and its impact on the error estimate.\\
\begin{notation}
    Throughout the article, we use standard notation on Lebesgue and Sobolev spaces and their norms. For a given subdomain $\Omega_1 \subseteq \Omega$, let $|\cdot|_{1,\Omega_1}$, $\|\cdot \|_{1,\Omega_1}$, $\|\cdot\|_{\infty}$, and $\|\cdot\|_{0,\Omega_1}$ denote the $H^1(\Omega_1)$-semi-norm, the standard $H^1(\Omega_1)$-norm, the $L^\infty$-norm, and the standard $L^2(\Omega_1)$-norm, respectively. The scalar product $(\cdot, \cdot)_{\Omega_1}$ is the $L^2$ inner product on the subdomain $\Omega_1$. We will omit the subscript $\Omega_1$ if it is the full domain, i.e., $\Omega_1=\Omega$. We write $a \lesssim b$ if there is a generic constant $C$ (independent from the discretization and multiscale parameters) such that $a\leq C b$. We also denote the derivative of the coefficient $\alpha$ w.r.t. the second argument by $\alpha_s$. 
 
\end{notation}

\subsection{Problem formulation}

In this paper, we consider the following quasilinear nonmonotone elliptic problem 
\begin{equation}\label{PDEPronlem}
       \begin{split}
           -\nabla\cdot{(\alpha(x,u)\nabla{u})}&=f \;\;\text{in}\;\;\Omega,  \\ u&=0 \;\; \text{on} \;\; \partial \Omega.
        \end{split}
\end{equation}
We assume homogeneous Dirichlet boundary conditions for simplicity. Other types of boundary conditions can be treated in a very similar manner. The matrix-valued coefficient $\alpha(x,u) \in \mathbb{R}^{d\cross d}$ encodes the properties of the material, in particular, we implicitly assume that it shows rapid oscillations on fine spatial scales. Therefore, we do not assume more than $L^\infty$-regularity of $\alpha$ w.r.t.~$x$, but some additional assumptions w.r.t.~the nonlinearity are required to guarantee the well-posedness of problem~\eqref{PDEPronlem}.

\begin{assumption}\label{assum_1} Suppose that

\begin{itemize}
    \item $\alpha(x,s)$ is uniformly Lipschitz continuous w.r.t.~the second argument, i.e., there exists  $\Lambda_0 >0$ such that 
    \begin{equation}
       \begin{split}
       \abs{\alpha(x, s_1)-\alpha(x,s_2)}\leq \Lambda_0 \abs{s_1-s_2} ,   \;  \;\forall x \in \Omega, \; \; s_1, s_2 \in \mathbb{R}.
       \end{split}
    \end{equation}
  
    \item $\alpha(x,s)$ satisfies the uniform ellipticity and boundedness conditions, i.e., there exist $\lambda >0$, and $\Lambda_1>0$ such that
    \begin{equation}\label{ellip}
                  {\alpha(x, s)\psi\cdot\psi} \geq \lambda |\abs{\psi}|^2, \quad \text{and} \quad \|\alpha(x,s)\psi\|\leq \Lambda _1 |\abs{\psi}| \;\; \forall  x \in \Omega , \;s \in \mathbb{R}, \;  \psi \in \mathbb{R}^d.
    \end{equation}
 
   \item For Fréchet linearization error analysis below, we impose an additional assumption that $\alpha(x,s)$ is twice differentiable w.r.t. the second argument. This assumption is not necessary to formulate our method.
\end{itemize}
\end{assumption}
In addition, we assume that $f \in L^2(\Omega)$. Now, consider the weak formulation of finding $u \in H_0^1(\Omega)$ that solves
\begin{equation}\label{WeakForm}
    \mathcal{A}(u;u,v) :=(\alpha(x,u)\nabla u,\nabla v) =(f,v)=:F(v) ,  \;\;\; \forall  v  \in H^1_0(\Omega).
\end{equation}
For any $q \in H_0^1(\Omega)$, the bilinear form $\mathcal{A}(q; \cdot, \cdot)$  inherits the ellipticity and boundedness conditions in $H^1_0(\Omega)$, i.e., there exists $\lambda>0 \;\text{and} \; \Lambda_1>0$ such that
\begin{align*}
    \lambda |w|_1^2 &\leq \mathcal{A}(q;w,w),  \; \; \forall q,w \in H^1_0(\Omega), \\
    \mathcal{A}(q;v,w) &\leq \Lambda_1 |v|_1|w|_1 , \; \; \forall q,v,w \in H^1_0(\Omega).
\end{align*}
Assumption \ref{assum_1}  ensures the existence and the uniqueness of the solution $u \in H^1_0(\Omega)$ of the weak problem~\eqref{WeakForm}.
The weak solution $u$ satisfies 
\begin{equation}\label{UnqueSolution}
    |u|_1 \leq C \|f\|_0,
\end{equation} 
where $C$ depends on the constants of~\eqref{ellip} in Assumption ~\ref{assum_1}, but is independent of the spatial variations of $\alpha$~\cite{chipot}. We refer to \cite{chipot, Douglas,Unique} for  proofs of existence and uniqueness of the weak solution. By a compactness argument, the authors first introduced  in \cite{Douglas} the existence of the weak solution $u$ of the formulation ~\eqref{WeakForm}. The authors in \cite{Unique} proved the existence of the weak solution $u$ as a weak limit of the Galerkin approximations $\{u_h\}_{h \xrightarrow{}0}$. In addition, the uniqueness is proven using the comparison principle. 
\subsection{Finite element method}

Consider a decomposition of the domain of interest $\Omega$ into a partition $\mathcal{T}_H$ of simplices or quadrilaterals. Let $K$ denote the elements of $\mathcal{T}_H$ and the corresponding diameter $H_K$, and define   \[H:=\max_{K \in \mathcal{T}_H}H_K.\] 
We assume that $\mathcal{T}_H$ is shape-regular and quasi-uniform. 

We emphasize that $\mathcal{T}_H$ is coarse in the sense that it does not recover the spatial variations of $\alpha$. Let $V_H$ be the standard lowest-order conforming finite element subspace of $H^1_0(\Omega)$. It consists of piecewise polynomials of total degree at most 1, if $K$ is a simplex,  or of partial degree at most 1 in each variable, if $K$ is quadrilateral. 
The finite element method seeks to find the solution $u_H \in V_H$ that satisfies
\begin{equation}\label{FEM_problem}
    \mathcal{A}(u_H;u_H,v_H)=(f,v_H) \;\; \forall v_H \in V_H.
\end{equation}

In practice, the arising nonlinear system is solved using an appropriate iterative method, e.g., the Newton method. The heterogeneity of $\alpha(x,u) \in \mathbb{R}^{d\cross d}$ w.r.t. $x$ needs to be considered when solving~\eqref{WeakForm} numerically. In order to have an accurate approximation solution that recovers the main properties of the solution $u$, the triangulation $\mathcal{T}_H$ should capture all the features of the medium. In practice the mesh size $H$ needs to be small enough. This assumption is crucial to guarantee the desired convergence order. However, that would lead to a computationally expensive method.

 In \cite{FEMquad}, the following  error estimate is derived  
 \[|u-u_H|_1 \leq C H^l,\]
given that $u \in H^{l+1}(\Omega)$ and $u_H$ is the FEM solution obtained with numerical integration. Under certain additional regularity assumptions, the convergence w.r.t. $L^2$ is shown to be \[\|u-u_H\|_0 \leq C_2 H^{l+1},\] where the constant $C_2$ is independent of the mesh size $H$.
It is for our case $l=1$, we need at least $u \in H^2(\Omega)$ for the above estimate to be hold. However, this is not always guaranteed as the coefficient $\alpha$ may be discontinuous w.r.t. $x$. Moreover, even if $u \in H^2(\Omega)$, the constants $C \text{ and } C_2$ depend on the spatial variations of $\alpha$, which is not desirable for multiscale coefficient.

 The existence of a solution to \eqref{FEM_problem} is guaranteed by the virtue of Brouwer fixed-point theorem. However, the uniqueness of the approximate solution is not always ensured, even though the weak problem itself admits a unique solution in $H^1_0(\Omega)$~\cite{uniqandnotunique}.   More precisely, the  uniqueness of  solution $u_{H} \in V_{H}$ to ~\eqref{FEM_problem} is only guaranteed for sufficiently fine meshes, i.e., there is $H_0$ such that the solution is unique for $H\leq H_0$. Under regularity  assumption of $u$ and sufficiently small  $H$, the uniqueness of the approximate solution with numerical quadrature is addressed in~\cite{FEMquad} for general order $l\geq 1$.  We refer to~\cite{FEMquad, uniqandnotunique, DouglasGalerkin} and the reference therein for further arguments and details. Here, we will employ the LOD method to achieve optimal convergence rate  without requiring more than $u\in H^1_0(\Omega)$.

\section{Multiscale method}\label{SectionMultiscaleMethod}
In this section, we introduce our multiscale method using basis functions which are computed from localized and linearized versions of problem~\eqref{WeakForm}. This is largely inspired by the multiscale method for quasilinear monotone problems in~\cite{Barbara}, which in turn is inspired by the Localized Orthogonal Decomposition (LOD)~\cite{LOD-Linear}. The key aim is to identify a low-dimensional subspace of $H^1_0(\Omega)$ with good approximation properties for the multiscale problem. This is done by a decomposition of $H_0^1(\Omega)$ into coarse and fine orthogonal subspaces. In the nonlinear case, we still have to identify the appropriate form of orthogonality, see below. Next, we discuss the key aspects of the LOD approach. We describe the orthogonal decomposition of the space $H_0^1(\Omega)$ and the steps involved. In addition, we outline the linearization procedures used to address the nonlinearity when solving the correction problems. That is required to construct the orthogonal multiscale space linearly. Finally, we explain the localization strategies employed to construct the space of interest locally. 
\subsection{(Orthogonal) Decompositions}
The first attempt to decompose $H_0^1(\Omega)$ into   subspaces depends mainly on using a suitable interpolation operator that satisfies certain properties. More precisely, let \[I_H: H^1_0(\Omega)\xrightarrow{}V_H\] be a bounded local linear projective operator, i.e., $I_H(v_H)=v_H \; \forall v_H \in V_H.$ Additionally, $I_H$ satisfies the following stability and approximation properties 
\begin{equation}\label{Interpolation_properties}
    |I_Hv|_{1,K}   \lesssim |v|_{1,N(K)}, \qquad
    \|v-I_Hv\|_{0,K}  \lesssim H|v|_{1,N(K)}, 
\end{equation}
 for all $v \in  H_0^1(\Omega)$. The subdomain $N(K)$ is the union of elements that have  nonempty intersection with the element $K$. 
 An example of projective interpolation that is commonly used in the LOD literature is defined as follows 
\[I_H: H_0^1(\Omega)\xrightarrow{\Pi_H}S^1(\mathcal{T}_H)\xrightarrow{E_H}V_H.\]
The operator $\Pi_H$ is defined to be the $L^2$ projection onto the elementwise affine functions $S^1( \mathcal{T}_H )$, and $E_H$ is an averaging operator. This choice is not restrictive, one can use any projective interpolation that satisfies the stability and approximation properties in~\eqref{Interpolation_properties}. Practically, we only require the kernel of the projection for the implementation of LOD. Given $I_H$ that satisfies the above properties  in~\eqref{Interpolation_properties}, the space $W=\operatorname{ker}I_H$ contains all the so-called fine-scale functions that cannot be captured in the finite element approximation space $V_H$. Due to the projection property of the interpolation $I_H$, the space $H_0^1$ can be uniquely decomposed as follows \[H_0^1=V_H\oplus  W.\] 
This decomposition does not yet depend on the problem formulation. For the linear case, the idea of the LOD is to consider another decomposition \[H_0^1=V_H^\mathrm{ms}\oplus  W\] that is orthogonal w.r.t. the energy scalar product of the problem.
If one tries to transfer the idea to the present nonlinear case, one might try to define the space $V_H^\mathrm{nl,ms}$ via
    \begin{equation}\label{Orthogonality}
        \mathcal{A}(v_H^\mathrm{nl,ms};v_H^\mathrm{nl,ms},w)=0 , \; \;  \forall v_H^\mathrm{nl,ms} \in V_H^\mathrm{nl,ms},  \;   \; \forall w \in W.  
    \end{equation}
However, this is a nonlinear ``space'' because of the nonlinearity of $\mathcal A$. This can, in particular, be seen when one tries to write it in the usual form $V_H^\mathrm{nl,ms} =(\operatorname{id}-Q^\mathrm{nl})V_H$, where $Q^\mathrm{nl}: V_H \xrightarrow{}W$ is a nonlinear so-called correction operator. The nonlinearity of $Q^\mathrm{nl}$ as well as $V_H^\mathrm{nl,ms}$ complicates the multiscale method due to the coupling of the nonlinear problem~\eqref{Orthogonality} to construct the space and the resulting nonlinear problem from the Galerkin approach. In particular, it is not straightforwardly clear that such a multiscale method is well-defined. Instead, we will discuss a linearized construction of multiscale basis functions in the following.

\subsection{Linearizations}\label{linearizaion}

Here, we introduce two possible linearization methods for $\mathcal A$ that we will use to formulate linearized corrector problems in the next subsection. The linearization approximation of the nonlinearity of $\alpha(x,v)\nabla v$ is of the following general abstract form \[ \alpha(x,v)\nabla v \approx \mathbf{A}_L(x,p^*,\nabla v)= \alpha_L(x,p^*)\nabla v + v\beta(x,p^*)+C_L(x,p^*).\] 
Note that the above approximation is linear with respect to $\nabla v$ and the function $p^* \in H^1_0(\Omega)$ is an a priori chosen, fixed function that we call the linearization point. Two possible techniques that we consider in this article are
\begin{itemize}
    \item  \textbf{Kačanov-type} linearization, which ``freezes'' the nonlinearity  in  $\alpha$ at the linearization point $p^*$, i.e.,
    \[\mathbf{A}_L(x,p^*,\nabla v)= \alpha(x,p^*)\nabla v.\] 
    
    This means that $\mathcal{\alpha}_L(x,p^*) =\alpha(x,p^*)$, $\beta(x,p^*)=0,$ and $C_L(x,p^*)=0$ in the abstract form.
  \item \textbf{Fréchet derivative} linearization, which is inspired by the Newton method. For general $v \in H^1_0(\Omega)$, the Fréchet derivative of the nonlinear coefficient $\alpha$ at the linearization point $p^*$ in the direction of $v-p^*$ reads
     
\[ \mathcal{F}(p^*)(v-p^*)=\alpha(x,p^*)\nabla (v-p^*) + (v-p^*)\alpha_s(x, p^*)\nabla p^*.\]
 
To simplify the presentation and the computation of the correctors,  the constant $\alpha(x,p^*)\nabla p^*$ can be omitted   in   (\ref{remove_the_constant}), so we introduce the following linearization formula 
\begin{align*}
   \mathbf{A}_L(x,p^*,\nabla v)&=\alpha(x,p^*)\nabla v+ (v-p^*)\alpha_s(x, p^*)\nabla p^*.
\end{align*}
In abstract form, the coefficients correspond as follows: $\mathcal{\alpha}_L(x,p^*) =\alpha(x,p^*)$, $\beta(x,p^*)=\alpha_s(x, p^*)\nabla p^*$, and $C_L(x,p^*)=0$. 

\end{itemize}

\subsection{Linearized correction operators}
We proceed by applying the  linearization strategies  presented in the previous section. 
%\begin{assumption}\label{assump2} Suppose that 
 %    \begin{itemize}
  %   \item   $\alpha_L(x,p^*) \in L^\infty(\Omega,\mathbb{R}^{d\times d})$ and $C_L(x,p^*) \in L^2(\Omega,\mathbb{R}), \; \; \forall p^* \in  H^1_0(\Omega)$.
  %   \item there exists $0<\lambda_L<\Lambda_L$ such that 
   %                     \[\lambda_L |\psi|^2\leq \alpha_L(x,p^*)\psi\cdot \psi \leq \Lambda_L|\psi|^2 ,\; \;\forall x \in \Omega , p^* \in H_0^1(\Omega), \psi \in \mathbb{R}^d.\]
 %\end{itemize}
 %\end{assumption}. 
 Based on the linear approximations of  $\alpha$ at $p^*$, we introduce the following bilinear  
   form  
    \[\mathcal{A}_L(p^*,v_1,v_2)=(\mathbf{A}_L(x,p^*,\nabla v_1),\nabla v_2)_\Omega.\]
 We define the linearized corrector operator $Q: V_H\xrightarrow{}W$ such that it satisfies the following orthogonality 
\begin{equation}\label{LinearOrthogonality}
    \mathcal{A}_L( p^*, v_H-Qv_H,w)=0, \; \;\forall v_H \in V_H \;\text{and }\forall w \in W.
\end{equation}
Equivalently, equation (\ref{LinearOrthogonality}) can be expressed as 

\begin{equation}\label{remove_the_constant}
    \int_{\Omega} \mathbf{A}_L(x,p^*,\nabla v_H)\cdot \nabla w dx =\int_{\Omega} \mathbf{A}_L(x,p^*,\nabla Qv_H)\cdot\nabla w dx .
\end{equation}
The linearization $\mathbf{A}_L(x,p^*,\nabla v_1)$  incorporates our two linear models, the Kačanov-type and Fr\'echet derivative linearization.
Given that    $\beta(x,p^*)=0  \text{ and } C_L(x,p^*)=0$, Assumption~\ref{assum_1} directly yields the well-posedness of the corrector problem (\ref{remove_the_constant}) for the Kačanov-type linearization. The well-posedness in the case of Fréchet derivative will be addressed and analyzed in Subsection~\ref{weposdnesssubsection}.

The ideal multiscale method is now a Galerkin method based on $V_H^\mathrm{ms}=(\operatorname{id}-Q)V_H$, which is now a linear space, such that the method is indeed well-defined.
%(one the corrector problem are).
We still call this method ideal because the correction operator $Q$ has global support for $d>1$. In the linear case, the authors of ~\cite{LOD-Linear} proved that the correction $Q$ decays exponentially away from the element of interest, we refer to Proposition~\ref{my_proposition}. Note that these results can also be applied to our \emph{linearized} corrector problems. Hence, we localize the corrector $Q$ by truncating the domain of computation as detailed in the next subsection. 

%{In order to guarantee the well-posedness of the corrector problems, we require that the linear approximation $\mathbf{A}_L(x,p^*,\nabla v)$ satisfies Assumption \ref{assum_1} which are met since  $\alpha_L(x,p^*)$  are inherited from Assumption~\ref{assum_1} for the original problem.  In addition we assume that $\beta_L(x,p^*) \in L^2(\Omega,\mathbb{R}), \; \forall p^* \in  H^1_0(\Omega).$ Further assumptions will be discussed in subsection \ref{weposdnesssubsection}.}

 \subsection{Localization and practical method}
In this subsection we introduce the localization of the global correction problem~\eqref{LinearOrthogonality}. We first present some additional notations. 
 Let $N^k(T)$ denote the $k$-layer patch of neighboring elements of the element $T$. We inductively define it  as follows
\begin{align*}
N^0(T)=T, \quad 
N^k(T)=\bigcup_{\substack{K \in 
\mathcal{T}_H \\ \overline{K} \cap \overline{N^{k-1}(T)} \neq \emptyset} }K.
\end{align*}
The parameter $k$ represents the degree of localization of the corrector because it determines the size of the patch. The quasi-uniformity assumption of $\mathcal{T}_H$ ensures that the number of elements that belong to $k$-layer patches is bounded by a constant $C_\mathrm{ol}$ that only depends on $k$ in a polynomial manner, i.e., 
\begin{equation}\label{Uniform_regu}
\max_{T \in \mathcal{T}_H}|\{K \in \mathcal{T}_H:K \in N^k(T)\}| \leq C_\mathrm{ol}.
\end{equation}
Now, we define the localized correction operator as \[Q^k:V_H \xrightarrow{} W , \;\;\; Q^k=\sum_{T \in \mathcal{T}_H } Q_{T,k},\] where 
\[ Q_{T,k}:V_H \xrightarrow{}W(N^{k}(T)):=\{ w \in W : \text{supp}(w)  \subseteq\overline{N^k(T)}\}.\]
The local correction problem reads as follows: For $v_H \in V_H$,  find the correction $Q_{T,k}v_H \in W(N^k(T))$  that satisfies
\begin{equation}\label{Corr}
    (\mathcal{A}_L)_{N^k(T)}(p^*,Q_{T,k}v_H, w)=(\mathcal{A}_L)_{T}(p^*,v_H, w) , \;\forall w \in W(N^{k}(T)).
\end{equation}
Here, the bilinear forms on the left and right sides are restricted to the subdomains ${N^k(T)} \; \text{and} \; T$, respectively. 

The corrector problems~\eqref{Corr} are still posed on infinite dimensional subspaces $W(N^{k}(T))$, which need to be discretized. This is accomplished as usual by introducing a fine-scale mesh $\mathcal{T}_h$ of $\Omega$ with elements $K$ of diameter $h_K$. Let $h:=\max_{K \in \mathcal{T}_h}{h_K}$ such that $h\ll H$ resolves all features of $\alpha$. Let $V_h$ be the lowest order Lagrange finite element space associated with $\mathcal{T}_h$. Consider the following fine-scale space\[W_h(N^k(T)):= \{ w_h \in V_h\cap W : \operatorname{supp}(w_h) \subseteq \overline{N^k(T)}\}.\]
    In practice, the corrector problem~\eqref{Corr} is solved with the new fine-scale space $W_h(N^k(T))$.

To construct our multiscale space, we replace $Q$ by $Q^k$ and define \[V_{H,k}=(\operatorname{id}-Q^k)V_H=\{v_H-Q^kv_H: v_H \in V_H\}.\] Note that $V_{H,k}$ is constructed by correcting the basis of the space $V_H$. Finally, we formulate our Galerkin method on $V_{H,k}$ as follows: Find $u_{H,k} \in V_{H,k}$ that satisfies
\begin{equation}\label{local_probelm }
    \mathcal{A}(u_{H,k};u_{H,k},v_{H,k})=(f,v_{H,k}), \; \forall \;v_{H,k} \in V_{H,k}.
\end{equation}
Observe that we solve the nonlinear problem on the low-dimensional space $V_{H,k}$ of the same dimension as $V_H$. The localization procedure above affects only the space construction, but not the final problem. In practice, then again one uses an iterative method to solve the associated nonlinear system.

One can also consider a Petrov-Galerkin formulation of the LOD to reduce the communication between the correctors. We seek   the solution  $u^{PG}_{H,k} \in V_{H,k}$ such that \[\mathcal{A}(u^{PG}_{H,k} ;u^{PG}_{H,k} ,v_{H})=(f,v_{H}),  \; \forall v_{H} \in V_{H}.\]

The   solution $u_{H,k} \in V_{H,k}$ of ~\eqref{local_probelm }  is ensured to exists as in \cite{uniqandnotunique}. The error analysis below, however, does not require uniqueness of $u_{H,k}$, but holds for any discrete solution. In practice, we did not experience any issues with non-uniqueness.
 
 \subsubsection{Coercivity of local correction problems on  the space W}\label{weposdnesssubsection}
The effect of localizing the correction problems is quantified as follows. The proposition below  is proved in~\cite{LOD-Linear} for the case of linear elliptic problems, where the correction problems are also linear elliptic.  It also shows that the error between the corrector operator and its local version decays exponentially in the oversampling parameter $k$.
 \begin{prop}\label{my_proposition}
   Let Assumptions~\ref{assum_1}  be satisfied. Let $Q$ be the linear corrector defined in~\eqref{LinearOrthogonality} and its localized version $Q^k$ defined in~\eqref{Corr}. There exists a constant $0<\nu<1$ such that for any $v_H \in V_H$, the following inequality is satisfied 
   \[|(Q-Q^k)v_H |_1\lesssim C_\mathrm{ol}^\frac{1}{2} \nu^k |v_H|_1,\]
   where $C_\mathrm{ol}$ is the constant in~\eqref{Uniform_regu}.
\end{prop}
\textbf{Applicability of Proposition \ref{my_proposition} to problem~\eqref{Corr}: }
 Kačanov linearization model ends up with a linear elliptic problem that inherits the assumptions of coercivity and boundedness. As a result, the correction problem~\eqref{Corr} is well-posed. Moreover, the inherited ellipticity  ensures the validity of the exponential decaying error  for the correctors that are computed using  the Kačanov technique. For Fréchet derivative case, in~\cite{Maier2020}, the exponential decay of the correction operator is proved for general bilinear forms that satisfy an inf-sup condition satisfied on the fine-scale space $W$ and its subspace $W^c_{k,K}:=\{w \in W: \operatorname{supp(w)} \subseteq \Omega \backslash N^k(K)\}$. Therefore, this ensures the uniqueness of the solution for the correction problem. For our case when using the Fréchet-type linearization, we show the coercivity of $\mathcal A_L$ on the space $W$ for sufficiently small $H$. The same proof can also be applied for $W^c_{k,K}$.
 
 \begin{lemma}\label{lemma} 
 For a fixed $p^* \in H^1_0(\Omega)$, let $C_{p^*}$ be such that $\|\alpha_s(x,p^*)\nabla p^*\|_\infty \leq C_{p^*}$.  If $H$ is sufficiently small  such that $C_IC_{p^*}H<1$ with  $C_I$ is the constant from~\eqref{Interpolation_properties}, then the bilinear form $\mathcal{A}_L(p^*,w,w) $ defined by Fréchet derivative  is coercive.
 \end{lemma}
 \begin{proof}
     By Assumption \ref{assum_1},  $C_{p^*}$ can always be found and depends on the choice of the (fixed) linearization point $p^*$.
  By the ellipticity of $\alpha$, we obtain for any $w\in W$ that 
\begin{align*}
    |w|_1^2 &\lesssim \int_\Omega (\alpha(x,p^*)\nabla w +w\alpha_s(x,p^*)\nabla p^*)\cdot\nabla w dx -\int_\Omega w\alpha_s(x,p^*)\nabla p^*\cdot\nabla w dx \\
    & \lesssim (\mathcal{A}_L(p^*,w,w)  + \|\alpha_s(x,p^*)\nabla p^*\|_\infty \|w\|_0)|w|_1 \\
     & \lesssim (\mathcal{A}_L(p^*,w,w)  + \|\alpha_s(x,p^*)\nabla p^*\|_\infty \|w-I_Hw\|_0)|w|_1 \\
    &\lesssim \mathcal{A}_L(p^*,w,w)  + C_IC_{p^*}H |w|_1^2.
\end{align*}

 \end{proof}
 This implies the well-posedness of the correction problem \eqref{LinearOrthogonality} that is solved using Fréchet derivative as a linearization formula. Moreover, the exponential decay of $Q$ as in Proposition~\ref{my_proposition} holds.

\section{Numerical experiments}\label{SectionExperiment}

In this section, we test our multiscale method on several numerical experiments, illustrating the convergence with respect to the mesh size and the influence of the linearization strategy. The  experiments consider problem~\eqref{PDEPronlem} with the same spatial multiscale coefficient, but different nonlinearities inspired by models for the stationary Richards equation, see~\cite{exmaples}. 
We choose the computational domain $\Omega=[0,1]^2$ and the right-hand side $f \in L^2(\Omega)$ as
\begin{equation}\label{right_hand}
    f(x) = \begin{cases}
    0.1 & \text{if } x_2 \leq 0.1, \\
    1 & \text{otherwise}.
\end{cases}
\end{equation}
For the nonlinear diffusion coefficient, we choose $\alpha(x,u)=c(x)k(u)$, where the spatial part $c$ is shown in Figure~\ref{variation} and the model for $k$ is given in each subsection. The spatial coefficient $c$ is piecewise constant on a scale $\varepsilon=2^{-6}$ and exhibits a high-contrast channel. Our results are compared to a reference solution $u_h \in V_h$ obtained using a standard FEM on a fine mesh of size $h=2^{-7}$, which resolves all features of $c$. The local correction problems~\eqref{Corr} are solved on the same fine mesh%, cf. Remark~\ref{RemarkLod}
. As an interpolation operator in the definition of $W$, we use the $L^2$ projection. The coarse-scale mesh sizes are chosen as $H=2^{-1},2^{-2},\ldots ,2^{-6}$. We emphasize that $H$ generally does not resolve the fine-scale features of the solution $u$. For the implementation of the LOD method, we follow the algorithm in~\cite{TheCode} with four different oversampling parameters $k=1,2,3,4$. The code is available at https://github.com/Maherkh/LodNonmonotoneNonlinearPDE. In particular, the linear space $V_{H,k}$ is constructed only once, and then we solve the global nonlinear problem~\eqref{local_probelm } iteratively. Precisely, we use the Kačanov iterative scheme with tolerance $\text{tol}=10^{-12}$ and starting value $u_0=p^*$, where $p^*$ is the linearization point used for the correction problems and is discussed below. In all experiments, this converged within the maximum of 10 iterations. 
\begin{figure} 
    \centering
    \includegraphics[width=0.5\textwidth]{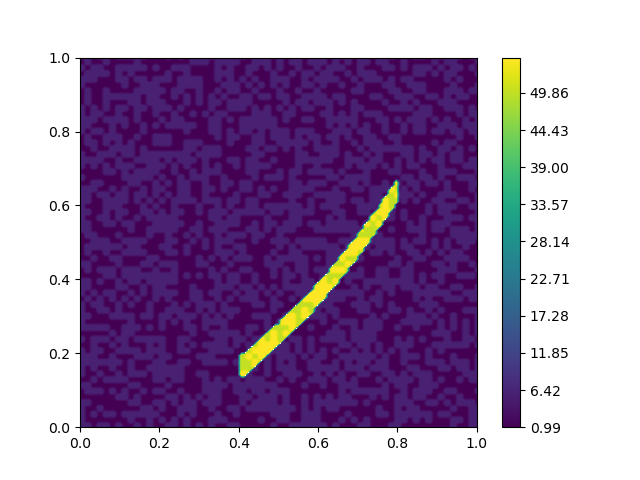}
    \caption{Spatial coefficient $c(x)$ }\label{variation}
\end{figure}
 In our experiments, we compare the performance of the two linearization methods introduced above. Furthermore,  we  examine  the impact of  different linearization points $p^*$ on the error behavior.  The chosen points are ranging from functions close to the analytical solution to functions further away from the analytical solution. 
 Precisely, we consider the following options: 
\[ p^* = 
\begin{cases}

    g(x)= 10x y (1 - x)  y  (1 - y),\\
    g_1(x)  =  0.5x y (1 - x)  y  (1 - y)e^{(5(x+y))},\\
    0 ,   \\
    
    u_H, & \text{FEM coarse solution, $H=\frac{1}{32}$,}   \\
    u_h, & \text{FEM reference solution solution, $h=\frac{1}{128}$,}\\
    \operatorname{ulod} & \text{LOD solution $H=\frac{1}{16}$, $k=4$, and $p^*=g_1.$}
\end{cases} \]
We provide a brief explanation of the choices of $p^*$ presented above. We use the coarse FEM solution $u_H$ because we believe it holds some information about the analytical solution. However, the size of the mesh $H$ does not recover the fine details of the analytical solution. Practically, the reference solution $u_h$ can not be used as a linearization point. However, it is used to investigate its impact on the linearization error. We tested the LOD solution $p^*=\operatorname{ulod}$ in order to check the effectiveness of an iterative LOD strategy on the improvement of the LOD discrete solution. Regarding $g \text{ and }g_1$, we choose them arbitrarily such that $g$ is of relatively small values and close to the solution. However, $g_1$ is obtained by scaling $g$ so that it deviates significantly away from the analytical solution. To study the convergence performance, we use the following relative errors
\[e_{\operatorname{LOD}}:= \frac{|u_h-u_{H,k}|_1}{|{u_h}|_1} \quad \text{and}\quad e_{H}:= \frac{\|{u_h-I_Hu_{H,k}}\|_0}{\|u_h\|_0},\]
called the relative upscaled error and the relative macroscopic error, respectively.
%According to Theorem~\ref{DecayingBound} and Corollary~\ref{Corollary}, 
Similar to the linear case and the monotone nonlinear problem in \cite{Barbara}, we expect $e_{\operatorname{LOD}}$ to converge linearly with respect to $H$, up to a linearization error, cf. Theorem~\ref{my_proposition}. The relative error $e_H$ is expected to converge at the same rate as the $L^2$-best approximation error as in Corollary \eqref{Corollary} below.
 
   \subsection{Exponential model}
\textbf{Exponential model 1: } Before we get into an in-depth experimental study,  we aim  to investigate the effect of iteratively solving the global nonlinear problem on the development of the relative error. To verify this, we apply our LOD approach on the so-called exponential model \[k(u)=\exp(4u),\]
and 
\[f(x) = \begin{cases}
    100 & \text{if } x_2 \leq 0.25, \\
    2 & \text{otherwise}.
\end{cases} \]
The corrector problem is linearized via Kačanov techniques at $p^*=0$. %and the correctors are not updated, i.e., the multiscale space $V_H^k$ is constructed only once.
In Figure \ref{errorvsiteration}, the plot demonstrates that the error is substantially reduced with   the number of iterations in particular for relatively smaller coarse mesh sizes and it starts to stagnates after few iterations. 

From this point onward, we perform our  experiments with the  right-hand side given in  \eqref{right_hand}. 
\begin{figure}
    \centering
    \includegraphics[width=0.55\linewidth]{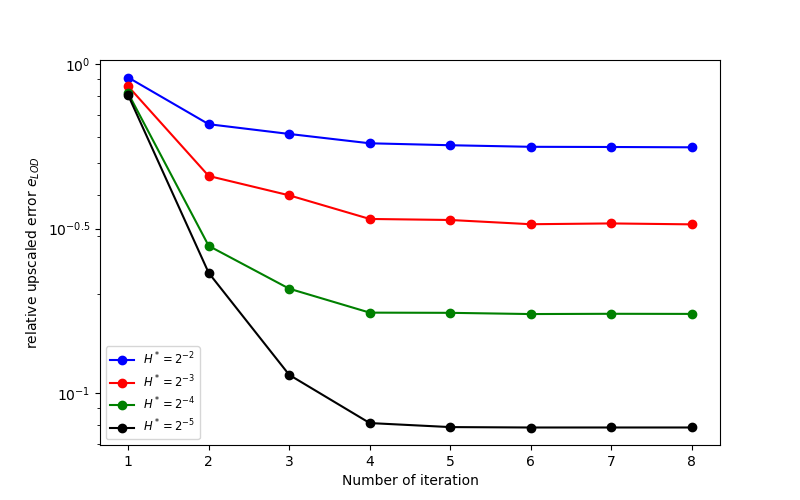}
    \caption{Relative upscaled error $e_{\operatorname{LOD}}$ with respect to the number of iterations obtained for a fixed oversampling parameter $k=4$.}
    \label{errorvsiteration}
\end{figure}

\textbf{Exponential model 2: } Here, we choose the the following  exponential model with $k(u)= \exp(2u)$. However, it is clear that $\alpha(x,s)$ does not satisfy Assumption~\ref{assum_1} for this model. Figure~\ref{ref_lod_solution_expo} shows the reference solution and an LOD approximation. Observe the influence of the high-contrast channel on the bottom right and the multiscale nature of both solutions arising from the rapid variations in $c(x)$. Both aspects are well captured by the LOD solution.
\begin{figure} [h]
    \centering
    % First image
    \begin{minipage}{0.43\textwidth}
       \centering      \includegraphics[width=1.1\textwidth]{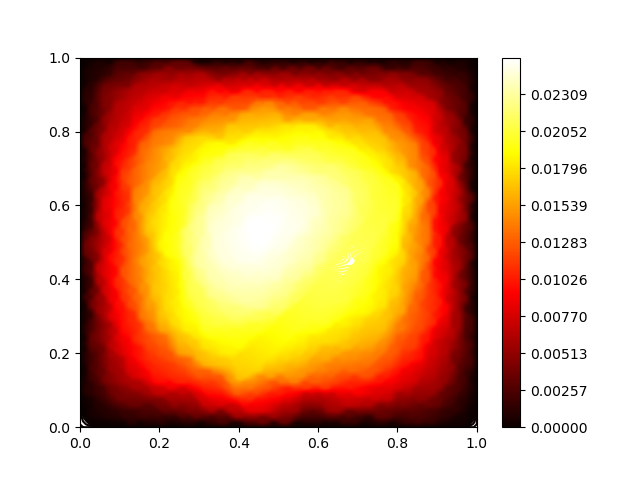}

    \end{minipage}
    \hfill
    % Second image
    \begin{minipage}{0.43\textwidth}
         \centering  
    \includegraphics[width=1.1\textwidth]{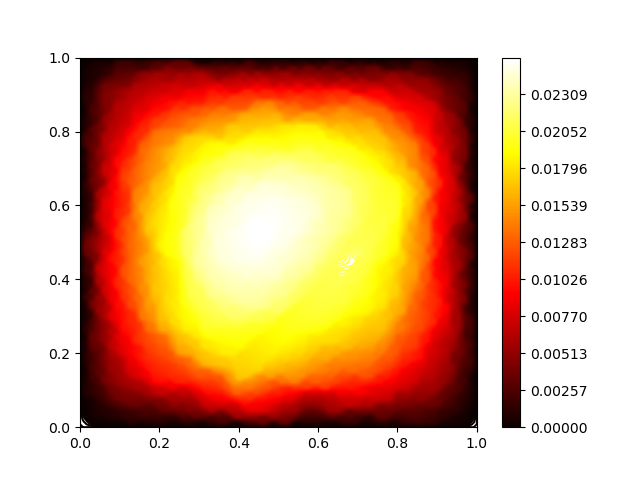}

    \end{minipage}
    \caption{Reference solution (left) and LOD solution (right) obtained by Kačanov linearization method for $H=2^{-4}$ and $k=4$.}\label{ref_lod_solution_expo}  
\end{figure}

\begin{figure}[ht]
    \centering
    % First image
    \begin{minipage}{0.49\textwidth}
         \centering  
    \includegraphics[width=1.1\textwidth]{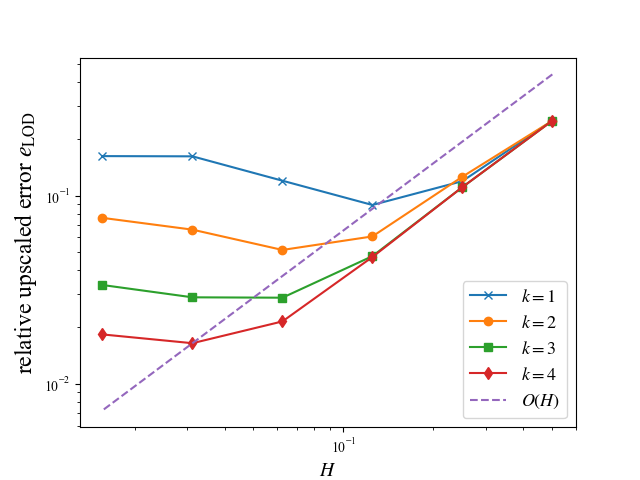}      \end{minipage}
    \hfill
    \begin{minipage}{0.49\textwidth}
       \centering  
    \includegraphics[width=1.1\textwidth]{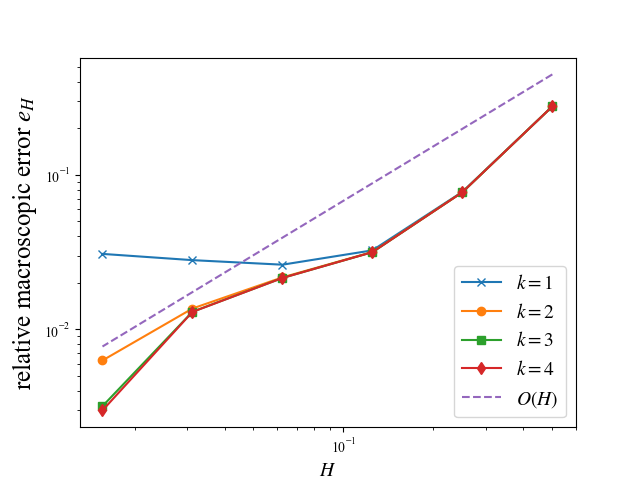}

    \end{minipage}
     
    % Second image
    \caption{Convergence history of $e_{\operatorname{LOD}}$ (left) and $e_H$ (right) for different oversampling parameters and fixed linearization point $p^{*}=u_H$ using Fréchet linearization.}\label{polp1}
\end{figure}

\begin{figure}[ht]
    \centering
    % First image
    \begin{minipage}{0.49\textwidth}
         \centering  
    \includegraphics[width=1.1\textwidth]{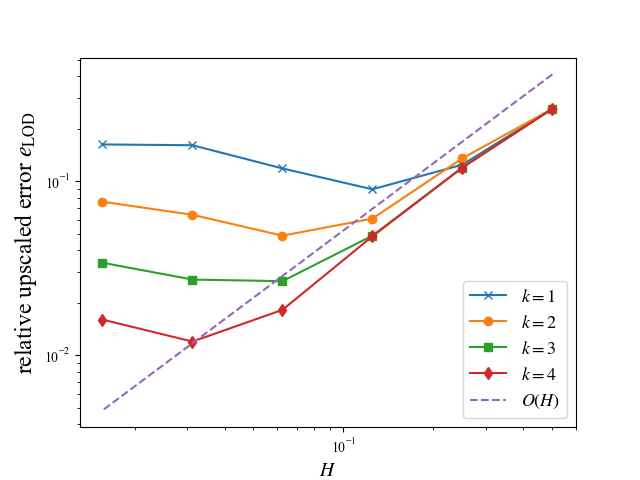}

    \end{minipage}
    \hfill
    \begin{minipage}{0.49\textwidth}
       \centering  
    \includegraphics[width=1.1\textwidth]{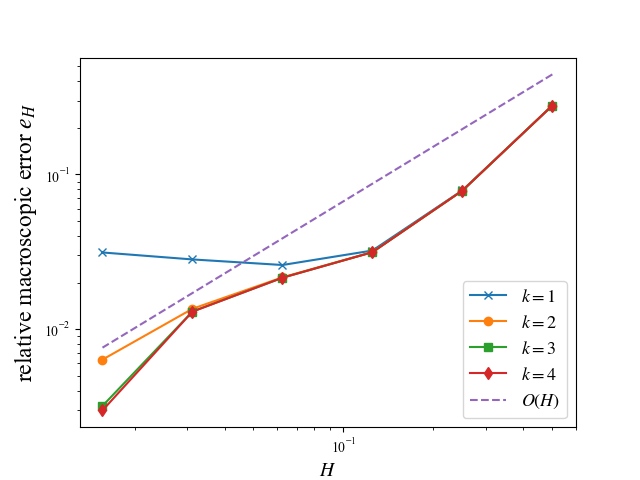}

    \end{minipage}
 
     \caption{Convergence history of $e_{\operatorname{LOD}}$ (left) and $e_H$ (right) for different oversampling parameters and fixed linearization point $p^{*}=u_H$ using Kačanov linearization.}\label{polp}
\end{figure}

Figures~\ref{polp1} and~\ref{polp} depict the convergence histories of our two error measures for the two linearization strategies and linearization point $p^*=u_H$. Regarding $e_{\operatorname{LOD}}$, we observe the predicted first order convergence rate. %from Theorem~\ref{DecayingBound}.
The saturation of the error curves corresponds to a domination of the localization error and sets in later for larger oversampling parameters. %\textcolor{red}{, which perfectly agrees with Theorem~\ref{DecayingBound} as well}.
Note that the overall convergence rate may even seem to be slightly better than linear convergence.
 We observe that the convergence rate $H$ of the relative macroscopic errors $e_H$ are very similar for all oversampling parameters $k\geq2$. %\textcolor{red}{corresponds to what we expect from Corollary~\ref{Corollary}}
 We cannot hope for more than linear convergence of the $L^2$-best approximation due to the low regularity of the solution caused by the discontinuities in $c$.
Comparing Figures~\ref{polp1} and~\ref{polp}, the two linearization strategies show almost the same performance. In particular, no impact of the linearization error is observed, probably because $p^*=u_H$ is close enough to both solutions $u$ and $u_{H,k}$. 
\begin{figure}[ht]
    \centering
    % First image
    \begin{minipage}{0.49\textwidth}
         \centering  
    \includegraphics[width=1.1\textwidth]{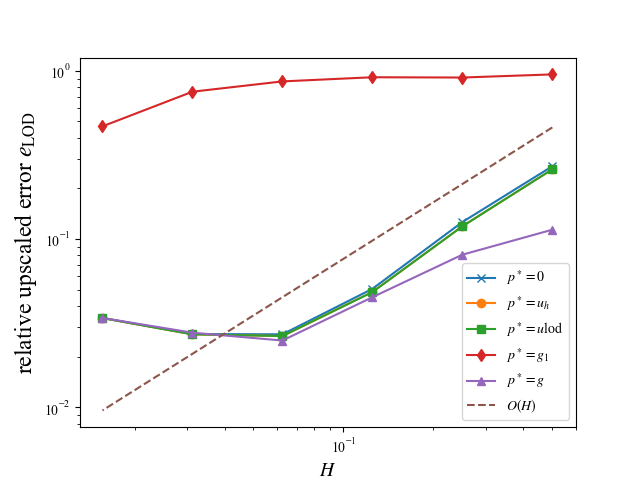} 
    %\caption{}
      
    \end{minipage}
    \hfill
    \begin{minipage}{0.49\textwidth}
       \centering  
    \includegraphics[width=1.1\textwidth]{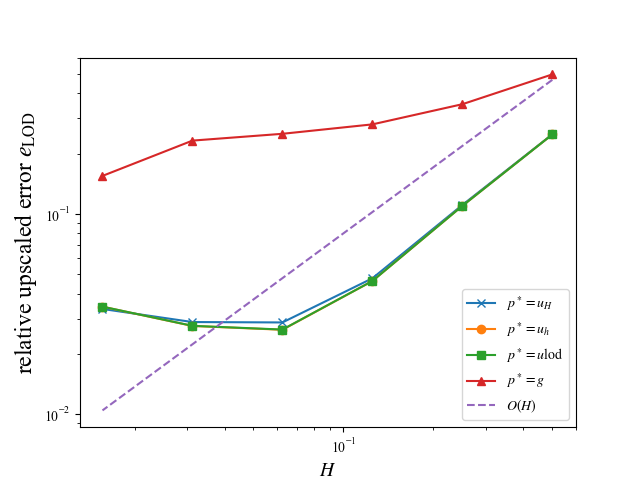 }

    \end{minipage}
    
    % Second image
    
    \caption{Convergence history of $e_{\operatorname{LOD}}$ for
    Kačanov linearization (left) and Fréchet method (right) with different linearization points and fixed oversampling parameter $k=3$.}\label{exp_comparion_the_same_methods}
\end{figure}
In Figure~\ref{exp_comparion_the_same_methods}, we study the impact of different linearization points on the convergence history, for both linearization techniques. For the Kačanov technique, the choices $p^* \in \{0,u_h,u_H,\text{ulod}\}$ all lead to very similar results, probably because all of them are close enough to the exact solution. Quantitatively, the choices $p^*= \operatorname{ulod}$ and $p^*= u_h$ perform best by a slight margin. Surprisingly, when $p^*=g$, we obtain even slightly better results, in particular for large mesh sizes $H$. %We believe that this might be related to the fact that $g$ represents the overall features of the solution already quite well. 
In contrast, using $p^*=g_1$ leads to large error values, failing to reach the first order of convergence with respect to $H$ due to the dominance of the linearization error. 
Finally, results for $p^*=g_1$ are not depicted for the Fréchet technique, because the error values blows-up. This is in line with the theory as the condition $C_IC_{p^*}H<1$ is violated, see the proof of Lemma \ref{lemma} in Subsection~\ref{weposdnesssubsection}. 
In this experiment, Fréchet linearization, hence, appears to be more sensitive with respect to the choice of the linearization point. It is important to emphasize that when using a linearization point that is close to the analytical solution, computing the correction only once seems to be sufficient, because no significant improvement is obtained when using $p^*=\operatorname{ulod}$. 
\begin{figure}[ht]
    \centering
    % First image
    \begin{minipage}{0.49\textwidth}
       \centering  
    \includegraphics[width=1.1\textwidth]{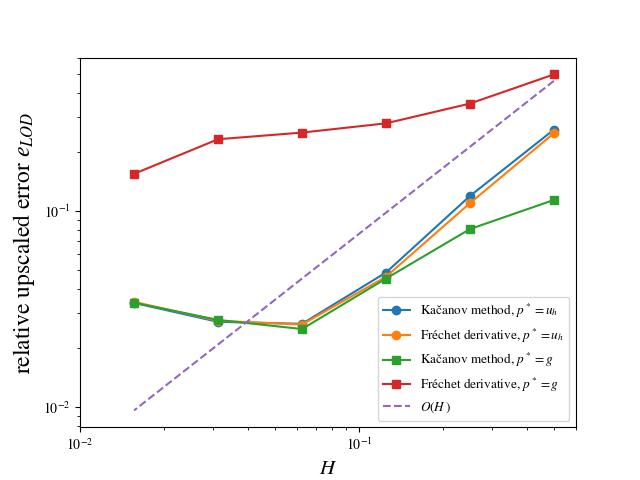 }
     
    \end{minipage}
    \hfill
    % Second image
    \begin{minipage}{0.49\textwidth}
         \centering  
    \includegraphics[width=1.1\textwidth]{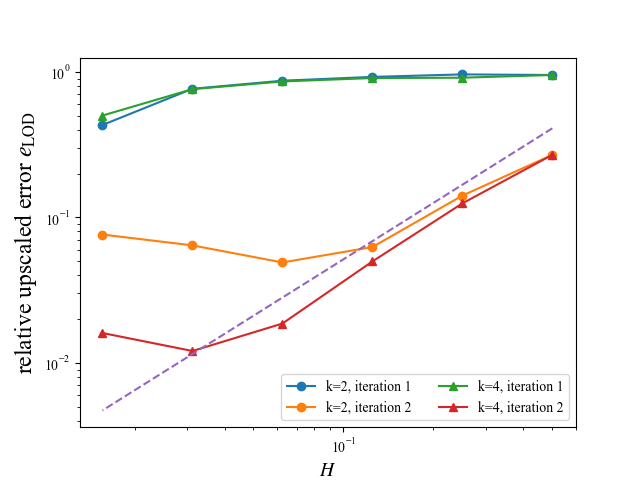}
    \end{minipage}
    \caption{Convergence history of $e_{\operatorname{LOD}}$ using different linearization points, $k=3$, with two different linearization methods (left), and for the iterative LOD (right).}\label{exp_compariondiffere}
\end{figure}

For even better visibility of these observations, we compare Fréchet and Kačanov linearization for two different choices of $p^*$ in one single Figure~\ref{exp_compariondiffere} (left). For $p^*=u_h$, we observe that both methods perform the same not only qualitatively but also quantitatively. The difference between Fréchet and Kačanov linearization goes to zero as the $p^*$ gets closer to the exact solution, which explains this observation. In contrast, when $p^*=g$, as discussed above, we observe a significant impact of the linearization method on the convergence. The Fréchet technique is clearly performing worse, which is explained by the fact that $\alpha_s(x,p^*)\nabla p^*$ becomes very large, so that the linearization error dominates the discretization error. In Figure~\ref{exp_compariondiffere} (right), we examine the Kačanov technique with $p^*=g_1$ in more detail. We observed in Figure~\ref{exp_comparion_the_same_methods} that $g_1$ is not a good choice and no convergence in the mesh size is obtained due to the large linearization error. However, taking the resulting bad LOD approximation as new linearization point $p^*$, the new LOD approximation recovers the expected convergence behavior. This shows the potential of iteratively computing a whole cascade of LOD solutions.  

\subsection{Van Genuchten model}\label{Van_Genuchten}
In our second experiment, we choose the nonlinearity according to the Van Genuchten model
 \[ k(u)=\frac{(1+\alpha |u|(1+(\alpha |u|)^2)^{-\frac{1}{2}})^2}{1+(\alpha|u|)^2},\quad \alpha=0.005 \]
and the spatial coefficient $c(x)$ as above. This nonlinearity grows more slowly for $|u|\to \infty$ than the exponential model and the values of $k(u)$ even remain bounded. Consequently, we expect that, in particular, the condition $C_IC_{p^*}H<1$ for the Fréchet linearization is satisfied for more linearization points $p^*$. 
\begin{figure}[ht]
    \centering
    % First image
     \begin{minipage}{0.49\textwidth}
         \centering  
    \includegraphics[width=1.1\textwidth]{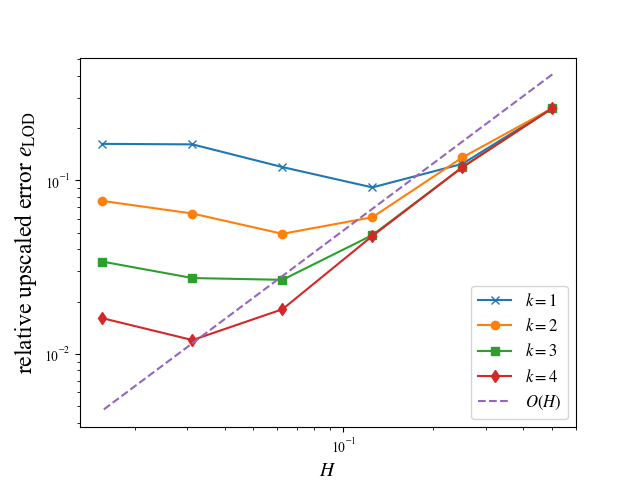}

    \end{minipage}
    \hfill
    \begin{minipage}{0.49\textwidth}
       \centering  
    \includegraphics[width=1.1\textwidth]{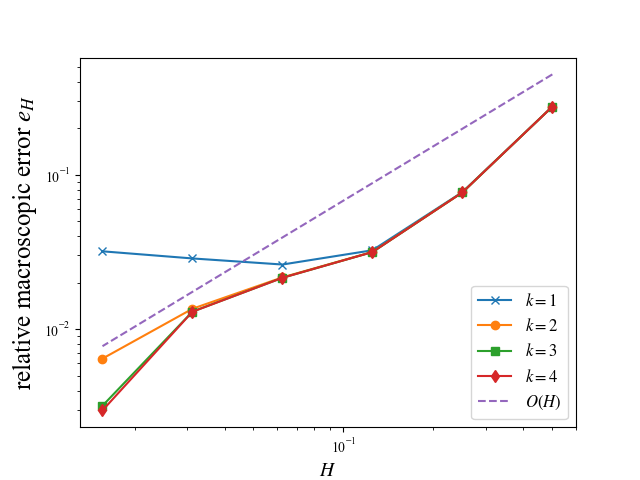} 
     
    \end{minipage}
   
   \caption{Convergence history of $e_{\operatorname{LOD}}$ (left) and $e_H$ (right) for different oversampling parameters and fixed linearization point $p^{*}=u_H$ using Fréchet linearization.}\label{ex1}
\end{figure}
As for the exponential model, we observe the expected rates for $e_H$ and $e_{\operatorname{LOD}}$ for both linearization strategies and $p^*\in\{0,u_h, u_H, \text{ulod}\}, $ cf. Figure~\ref{ex1}, where the experiment is tested for $p^*=u_H$. 
\begin{figure}[ht]
    \centering
    % First image
    \begin{minipage}{0.49\textwidth}
         \centering  
    \includegraphics[width=1.1\textwidth]{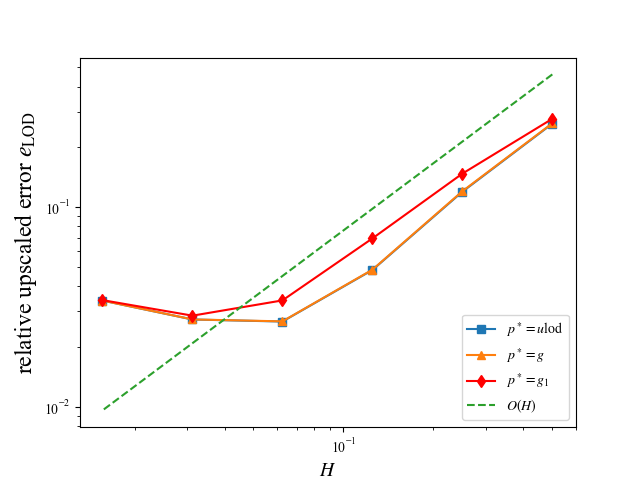}

    \end{minipage}
    \hfill
    \begin{minipage}{0.49\textwidth}
       \centering  
    \includegraphics[width=1.1\textwidth]{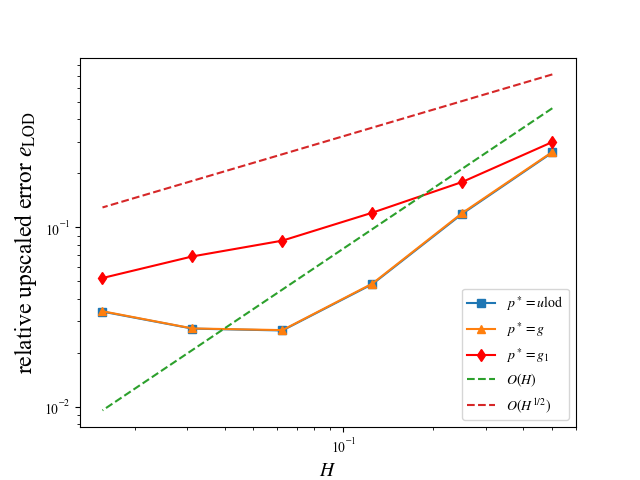} 
      
    \end{minipage}
  
    % Second image
    
    \caption{Convergence history of $e_{\operatorname{LOD}}$ for Kačanov method (left) and Fréchet method (right) with different linearization points and fixed oversampling parameter $k=3$.}\label{Different_methods_and_linearization_points}
\end{figure}
As a main difference from the exponential model, we focus in the following on the influence of the linearization point, in particular when $p^*=g$ or $p^*=g_1$. In Figure~\ref{Different_methods_and_linearization_points} (left), we observe that all linearization points perform qualitatively the same for the Kačanov technique. We especially emphasize that $p^*=g_1$, although having larger errors, now also shows the expected convergence rate in contrast to the exponential model. This different behavior is most probably due to the fact that  $k(u)$ yields small values for various choices of $p^*$ in the Van Genuchten model. 
In Figure~\ref{Different_methods_and_linearization_points} (right), we observe that all linearization points except for $g_1$ perform similar to the expected convergence rates also for the Fréchet linearization. Again, we emphasize the stark contrast to the exponential model, where now convergence in the mesh size occurs for $p^*=g$. The choice $p^*=g_1$ still yields the largest errors for the Fréchet linearization, but we now observe a convergence with respect to $H$, which is approximately of order $\frac{1}{2}$. This is in great contrast to the exponential model, where the method does not even converge. Due to the smaller values of $k$ in the Van Genuchten model, the condition $C_IC_{p^*}H<1$ appears to be satisfied here, but not for the exponential model as discussed above. %This occurs due the boundedness condition $\|\alpha(x,p^*)p^* \|_\infty<C$ is satisfied here. 
 To conclude, the Fréchet linearization still seems to be more sensitive with respect to the choice of $p^*$, but the smaller values of $k(u)$ in the Van Genuchten model reduce this effect and, in general, make all results more robust with respect to the choice of $p^*$.
\subsection{Combined models}
In this section, we seek to additionally investigate whether the specific structure of the diffusion coefficient has a significant influence on the robustness results we obtained previously.  Up to this point,  the numerical experiments discussed above have been performed on a coefficient of the form $\alpha(x,u)=c(x)k(u)$, in which the  experiments show that the method is rather robust with respect to the choice of $p^*.$ In Figure \ref{robustness}, we  extend the experiments into coefficients of different forms. 
In Figure \ref{robustness} (left), we test several linearization points for the  following combination of Van Genuchten models
\[\alpha(x,u)=c_1(x)k_1(u)+c_2(x)k_2(u), \]
where \[k_1(u)=\frac{(1+\alpha |u|(1+(\alpha |u|)^2)^{-\frac{1}{2}})^2}{1+(\alpha|u|)^2},\quad \text{ and } \alpha=5, \] and $k_2(u)$ is Van Genuchten model given in Section \ref{Van_Genuchten}. In addition,  in Figure \ref{robustness} (right) we again investigate the robustness for another combined model of exponential and Van Genuchten models
\[\alpha(x,u)=c_1(x)k_1(u)+c_2(x)\exp(2u).\]
Both $c_1 \text{ and } c_2$ are of the same structure  but with different parameters so that    $c_2$ has a higher contrast than $c_1$. Note that $c_2$ is the same spatial coefficient we used in the previous examples.
\begin{figure}[ht]
    \centering
    % First image
    \begin{minipage}{0.49\textwidth}
         \centering  
    \includegraphics[width=1.1\textwidth]{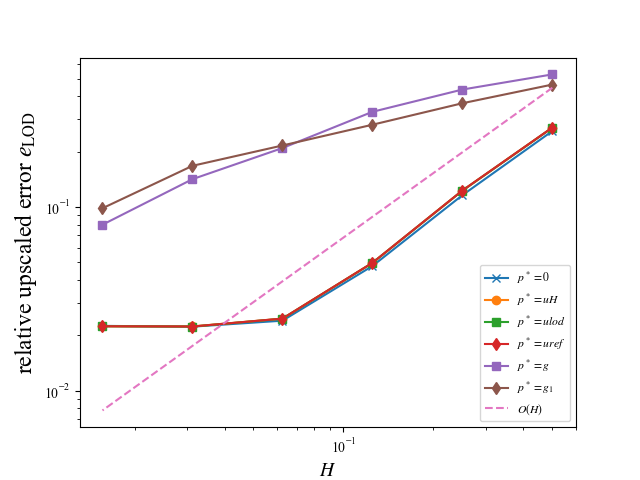}

    \end{minipage}
    \hfill
    \begin{minipage}{0.49\textwidth}
       \centering  
    \includegraphics[width=1.1\textwidth]{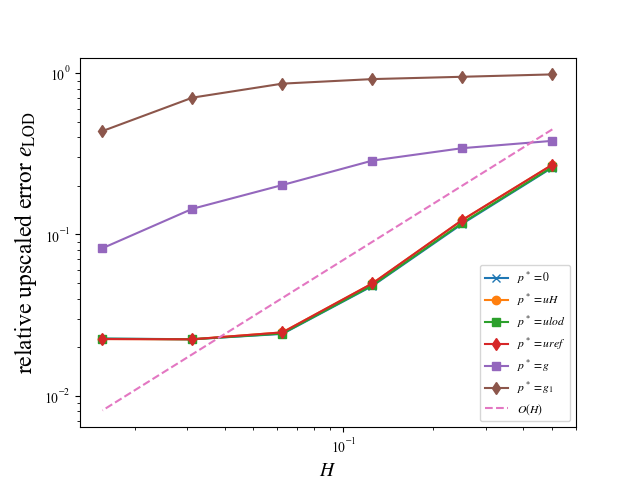} 
      
    \end{minipage}
  
    % Second image
    
    \caption{Convergence history of $e_{\operatorname{LOD}}$ of Kačanov method for several linearization points and fixed oversampling parameter $k=3$ applied on $\alpha(x,u)$ of combined models.}\label{robustness}
\end{figure}
The numerical results in Figure \ref{robustness} still show robustness of the numerical technique regardless of the form of the diffusion coefficient.

 \section{Theoretical justification via    error estimates}\label{SectionError}
 This section is devoted to a theoretical error analysis of the proposed approach, which, to some extent, theoretically interprets the practical results discussed earlier. We present the error analysis of the multiscale approach introduced in Section~\ref{SectionMultiscaleMethod}. Specifically, we outline the analysis for the two linearization strategies used in the correction computations, demonstrate the behavior of the linearization method, and examine the impact of the linearization points on the convergence rate tested in the numerical experiments above. %Assume that we have a fixed linearization point $p^* \in H^1_0(\Omega) $ and fixed tensor $\alpha(x,p^*)$ w.r.t. the second argument. %We essentially split the error of the multiscale method into a localization contribution and the error of the ``ideal'' method with $Q$. 
 We need to additionally quantify the error from the linearization of the correction problems. Given $ v \in H_0^1(\Omega)$ and for a fixed linearization point $p^* \in H_0^1(\Omega)$, we define the linearization error corresponding to the nonlinear form $\alpha(x,v)\nabla v $ as 
\[\eta(v) = \sup_{w \in H_0^1(\Omega), |w|_1=1}\Bigl|\int_\Omega[\alpha(x,v)\nabla v-\mathbf{A}_L(x,p^*,\nabla v)]\cdot\nabla w dx\Bigr|.\]
We can now formulate our main result.
\begin{theorem}\label{DecayingBound}
 Let Assumptions~\ref{assum_1}  be satisfied. Let $u$ be the solution to~\eqref{WeakForm}, and let $u_{H,k}$ be the multiscale solution to~\eqref{local_probelm }. Then it holds that
\[ |u-u_{H,k}|_1 \lesssim (H+{C_\mathrm{ol}^\frac{1}{2} \nu^k})\|f\|_0+\eta(u)+G(u,u_{H,k}, p^*),\]
where, given $v_{H,k}=(\operatorname{id}-Q^k)I_H u$, we define $G(u,u_{H,k}, p^*)$ as
\begin{align*}
    G(u,u_{H,k}, p^*) &:=\|(\alpha(x,p^*) -\alpha(x,u))\nabla v_{H,k}+(\alpha(x,u_{H,k})-\alpha(x,p^*))\nabla u_{H,k}\|_0.
\end{align*}
\end{theorem}

A detailed proof of this theorem   and the error analysis will be presented below. The theorem shows that we achieve a convergence of first order up to the linearization error. This rate has been observed numerically in Section \ref{SectionExperiment} and  also compatible with the results obtained in~\cite{Barbara} for the monotone nonlinear problem up to  linearization error. In the corollary below, we estimate the error for the finite element part of the LOD solution, namely, $I_Hu_{H,k}$.
\begin{corollary}\label{Corollary}
     Let Assumptions~\ref{assum_1}  be satisfied. Let $u$ be the solution of~\eqref{WeakForm} and let $u_{H,k}$ be the LOD solution of~\eqref{local_probelm }. Then we have 
     \begin{equation}\label{L2Estimate}
         \|u-I_Hu_{H,k}\|_0\lesssim  \|u-I_Hu\|_0+\|u-u_{H,k}\|_0 +H|u-u_{H,k}|_1.
     \end{equation}
     %\[\|u-I_Hu_{H,k}\|_0\lesssim H |u|_1 +|u-u_{H,k}|_1.\]
 \end{corollary}

 \begin{proof}
By triangle inequality and the approximation property of $I_H$ we have that
\begin{align*}
    \|u-I_Hu_{H,k}\|_0 &\leq \|u-I_Hu \|_0+\|I_Hu-I_Hu_{H,k}\|_0.
    %&\lesssim  H|u|_1+\|I_Hu-I_Hu_{H,k}\|_0.
\end{align*}
 Using the stability and approximation of the interpolation $I_H$, and Poincaré inequality, we obtain for the second term that
\begin{align*}
  \|I_Hu-I_Hu_{H,k}\|_0 &\lesssim \|u-u_{H,k}-(u-u_{H,k})+I_Hu-I_Hu_{H,k}\|_0, \\
  &\lesssim \|u-u_{H,k}\|_0+H|u-u_{H,k}|_1. \qedhere
\end{align*}
\end{proof}
The error estimate of Corollary~\ref{Corollary} can be further refined if an  $L^2$ estimate for $\|u-u_{H,k}\|_0$ is available (it should be $O(H^2)$ by Aubin-Nitsche trick). Hence, with the estimate of the first and last term in~\eqref{L2Estimate} and if $k\approx |\operatorname{log}(H)|$, it yields an estimate that is at least of the first order up to the linearization error and even of order $H^2$ if $u$ is sufficiently regular. 
Furthermore, the term $\|u-I_Hu\|_0$ can be estimated against the $L^2$-best-approximation error of $V_{H}$ if we assume $L^2$ stability of $I_H$, i.e., $\|I_Hv\|_0 \lesssim\|v\|_0 \; \forall v_H \in V_H$, the estimate~\eqref{L2Estimate} can be rewritten as
\[\|u-I_Hu_{H,k}\|_0 \lesssim\inf_{v_H \in V_H}\|u-v_H\|_0+\|u-u_{H,k}\|_0.\]
The $L^2$-best-approximation error converges at least $O(H)$ for $u\in H^1_0(\Omega)$, and even better rates are possible if we assume more regularity of $u$. 
 
In the following analysis, we split our explanations into two subsections based on the linearization technique. We first present the analysis for the case of using Kačanov technique and then proceed with the analysis for the case of Fréchet derivative. 

\subsection{Error estimate for Kačanov-type linearization}

We dedicate this section to the error analysis  when using  Kačanov-type linearization technique.
\begin{proof}[Proof of Theorem~\ref{DecayingBound}] 
By triangle inequality, we have for any $v_{H,k}\in V_{H,k}$
\begin{equation}\label{IneqProof}
    |u-u_{H,k}|_1 \leq |u-v_{H,k}|_1+|v_{H,k}-u_{H,k}|_1.
\end{equation}
We now specifically choose $v_{H,k}=(\operatorname{id}-Q^k)I_Hu=(\operatorname{id}-Q)I_Hu+(Q-Q^k)I_Hu$.
First, we bound the first term on the right-hand side of~\eqref{IneqProof} as follows
\begin{equation}\label{Step1}
    |u-v_{H,k}|_1 \leq |u-(\operatorname{id}-Q)I_Hu|_1 +|(Q-Q^k)I_Hu|_1.
\end{equation}
The second term is bounded using Proposition~\ref{my_proposition}, the stability assumptions of the interpolation, and~\eqref{UnqueSolution} to yield
\begin{equation}\label{correction}
    |(Q-Q^k)I_Hu|_1\lesssim {C_\mathrm{ol}^\frac{1}{2} \nu^k}\|f\|_0.
\end{equation}
 
To bound the first term of~\eqref{Step1}, let $\xi=(\operatorname{id}-Q)I_Hu$.  Using~\eqref{WeakForm} for the solution $u$,  the uniform ellipticity of $\alpha(x,v)$ for all $ v \in H^1_0(\Omega)$, the stability of the interpolation, and the orthogonality from~\eqref{LinearOrthogonality}, we obtain 
\begin{align*}
     |u-\xi|_1^2 &\lesssim \int_{\Omega} \alpha(x,p^*)\nabla(u-\xi)\cdot\nabla(u-\xi)dx =\int_{\Omega} (\alpha(x,p^*)\nabla u-\alpha(x,p^*)\nabla \xi)\cdot\nabla(u-\xi)dx  \\
     &= \int_{\Omega} (\mathbf{A}_L(x,p^*,u) -\mathbf{A}_L(x,p^*,\xi ) )\cdot\nabla(u-\xi)dx  \\
    & {=} \int_{\Omega} [\mathbf{A}_L(x,p^*,u)-\alpha(x,u)\nabla u+\alpha(x,u)\nabla u]\cdot\nabla(u-\xi)dx\\
   & \lesssim (H\|f\|_0+\eta(u))|(u-\xi)|_1. 
\end{align*}
In the last inequality, we have used that $u-\xi \in W$. Together with the approximation property of $I_H$, we obtain the following first order of convergence 
\begin{align*}
|(f,u-\xi)_0|&\leq \|f\|_0\|u-\xi \|_0=\|f\|_0\|u-\xi-I_H(u-\xi) \|_0 \\
&\leq H\|f\|_0|u-\xi|_1.
\end{align*}
 To estimate the second term of~\eqref{IneqProof}, observe that since both $u \text{ and } u_{H,k}$ are solutions to~\eqref{WeakForm} and~\eqref{local_probelm } respectively, we have that \[\int (\alpha(x,u)\nabla u-\alpha(x,u_{H,k})\nabla u_{H,k})\nabla \psi_{H,k}=0 \] for any test function $\psi_{H,k} \in V_{H,k}$.
This identity with the ellipticity and boundedness as well as the Lipschitz continuity of $\alpha(x,v)$ for all $v \in H^1_0(\Omega)$ yields
\begin{align*}
    |v_{H,k} -u_{H,k}|_1^2 &\lesssim \int_{\Omega} \alpha(x,p^*)\nabla(v_{H,k}-u_{H,k})\cdot\nabla(v_{H,k} -u_{H,k})dx \\
    &\leq \int(\alpha(x,p^*)\nabla v_{H,k} -\alpha(x,u)\nabla v_{H,k}+\alpha(x,u)\nabla v_{H,k}-\alpha(x,u)\nabla u\\
    &\qquad+\alpha(x,u_{H,k})\nabla u_{H,k} 
    -\alpha(x,p^*)\nabla u_{H,k}) \cdot\nabla(v_{H,k}-u_{H,k}) dx\\
    &\leq (G(u,u_{H,k}, p^*)+ \| \alpha(x,u)\nabla v_{H,k}-\alpha(x,u)\nabla u \|_0) {|v_{H,k}-u_{H,k}|_1}\\
    &\leq ( G(u,u_{H,k}, p^*) +\Lambda_1|u-v_{H,k}|_1 )|v_{H,k}-u_{H,k}|_1.
\end{align*}
% Inserting the previously obtained estimate for $|u-v_{H,k}|_1$, we have
%\begin{align*}
%    |v_{H,k}-u_{H,k}|_1 &\lesssim G(u,u_{H,k}, p^*)+\eta(u)+(H+{C_\mathrm{ol}^\frac{1}{2} \nu^k})\|f\|_0.
%\end{align*}
Finally, combining the estimates in \eqref{correction},    $|u-v_{H,k}|_1$, and $|v_{H,k}-u_{H,k}|_1$ yields
\begin{align*}
    |u-u_{H,k}|_1 & \lesssim (H+{C_\mathrm{ol}^\frac{1}{2} \nu^k})\|f\|_0 +\eta(u)+ G(u,u_{H,k}, p^*). 
\end{align*}

\end{proof}

\subsection{Error estimate for Fréchet derivative linearization}\label{Section3.2}

We now estimate the error in~\eqref{IneqProof} in the case of linearizing the correction computations with the Fréchet derivative. 

\begin{proof}[Proof of Proposition~\ref{DecayingBound}] 
Similar to the proof for part (a), we have~\eqref{IneqProof} and choose \[v_{H,k}=(\operatorname{id}-Q^k)I_Hu=(\operatorname{id}-Q)I_Hu+(Q-Q^k)I_Hu.\]
First, we bound the first term on the right hand side of the inequality~\eqref{IneqProof} as 
\begin{equation}\label{Step1b}
    |u-v_{H,k}|_1 \leq |u-(\operatorname{id}-Q)I_Hu|_1 +|(Q-Q^k)I_Hu|_1.
\end{equation}
As discussed above, Proposition~\ref{my_proposition} holds for sufficiently fine $H$ so that we bound the second term similar to the Kačanov case.
 
Recall $\xi=(\operatorname{id}-Q)I_Hu$.
To bound the first term of~\eqref{Step1b}, we use the uniform ellipticity of the linearized tensor $\alpha$, the stability of the interpolation, and the orthogonality from~\eqref{LinearOrthogonality} and obtain
\begin{align*}
     |u-\xi|_1^2 &\lesssim \int_{\Omega} (\alpha(x,p^*)\nabla(u-\xi)+\alpha_s(x,p^*)\nabla p^*(u-\xi))\cdot\nabla(u-\xi)dx\\ &\qquad-\int_{\Omega} \alpha_s(x,p^*)\nabla p^*(u-\xi)\cdot\nabla(u-\xi)dx,   \\
     &=\int_{\Omega} (A_L(x,p^*,u) - \alpha(x,u)\nabla u+\alpha(x,u)\nabla u )\cdot\nabla(u-\xi)dx \\&\qquad-\int_{\Omega} \alpha_s(x,p^*)\nabla p^*(u-\xi)\cdot\nabla(u-\xi)dx \\
   & \lesssim H\|f\|_0|u-\xi|_1+\eta(u)|u-\xi|_1 + \|\alpha_s(x,p^*)\nabla p^*\|_\infty \|u-\xi\|_0|u-\xi|_1, \\
   & \lesssim (H\|f\| +\eta(u)  + \|\alpha_s(x,p^*)\nabla p^*\|_\infty \|u-\xi-I_H(u-\xi)\|_0)|u-\xi|_1, \\
   & \lesssim (H\|f\|_0+\eta(u) +C_IC_{p^*}H|u-\xi|_1)|u-\xi|_1.
\end{align*}

As mentioned before, we assume $H$ to be sufficiently fine in the sense of $C_IC_{p^*}H<1$, so we obtain 
\begin{align*}
     |u-\xi| &\lesssim H\|f\|_0+\eta(u). 
\end{align*}
To bound the second term in~\eqref{IneqProof}, we follow the same steps of the proof of the Kačanov case and get 
\begin{align*}
    |v_{H,k}-u_{H,k}|
     &\lesssim     H\|f\|_0 + \eta(u)+ G(u,u_{H,k}, p^*). %+ {C_\mathrm{ol}^\frac{1}{2} \nu^k}\|f\|_0
\end{align*}
Combining  the estimates we obtain above with \eqref{correction} yields
\[ |u-u_{H,k}| \lesssim( H+{C_\mathrm{ol}^\frac{1}{2} \nu^k})\|f\|_0+ \eta(u)+G(u,u_{H,k}, p^*) . \]

\end{proof}
%\subsection{Linearization error analysis}
\subsection{Linearization error and the choice of linearization point}
 In the LOD   we try to build an approximation space that is problem-adapted. For the nonlinear problem, it is of importance to choose a suitable linearization point, in order to to be able to transfer the idea of the correction computation suggested for the linear problem. Since $u$ is not known a priori,  we try to choose $p^*$ close to the solution so that $V^ms_{H,k}$ has good approximation properties. In the next lemma we discuss possible estimate of the linearization errors of Proposition \ref{my_proposition}, specifically $\eta(u) \text{ and } G(u,u_{H,k},p^*).$

 \begin{lemma}
     Given the linearization formulas in Section \ref{linearizaion} of both Kačanov-type  and Fréchet derivative linearizations, then we have the following upper bounds
     \begin{enumerate}[label=(\alph*)]
    \item Kačanov-type linearization
    \[\eta(u) \lesssim \| p^* - u\|_\infty\|\nabla u\|_0.\]
    
    \item Fréchet derivative linearization
    \[\eta(u)\leq \|(\alpha(x,q^*)\nabla q^*)_{ss}\|_\infty\frac{\|u-p^*\|^2_0}{2}.\]
%\begin{align*}
 %  G(u,u_{H,k}, p^*)& \lesssim {\|u-p^*\|_\infty} \|\nabla u\|_0+ \|\nabla u_{H,k}\|_0 {\|u_{H,k}-p^*\|_\infty}\\
  %  &+ \|(\alpha_{ss}(x,q^*)\nabla q^*)\|_\infty{{\|u-p^*\|^2_0}}, 
  %  \end{align*}
where $q^* \in H^1_0(\Omega)$ is an intermediate function, i.e., $q^*=p^*+t(u-p^*)$, for some $t \in [0,1]$.

\end{enumerate} 
In addition,  we have the following estimate for $G(u,u_{H,k},p^*)$ such that
\begin{align*}
    G(u,u_{H,k}, p^*)& \lesssim\  {{\|u-p^*\|_\infty}}\|\nabla u\|_0 + \|\nabla u_{H,k}\|_0{\|u_{H,k}-p^*\|_\infty}. 
    \end{align*}
 \end{lemma}
  \begin{proof} By Lipschitz continuity of $\alpha$, we obtain the following linearization error bounds.\\
\textbf{Kačanov-type linearization error: }
  \begin{align*}
    \eta(u)&\leq\|\alpha(x,p^*)\nabla u -\alpha(x,u)\nabla u\|_0\\
    &\leq\|\alpha(x,p^*)-\alpha(x,u)\|_\infty\|\nabla u\|_0\\
    & \lesssim \| p^* - u\|_\infty\|\nabla u\|_0.
\end{align*}
 \textbf{Fréchet derivative linearization error: } Using the Taylor series expansion on $\alpha(x,u)\nabla u$, we obtain 
 \begin{flalign*}
\eta(u)&\leq \|\alpha(x,p^*)\nabla p^*+\alpha(x,p^*)\nabla (u-p^*) + (u-p^*)\alpha_s(x, p^*)\nabla p^*-\alpha(x,u)\nabla u\|_0 &&\\
    &\leq \|\alpha(x,p^*)\nabla p^*+\mathcal{F}(p^*)(u-p^*)-\alpha(x,p^*)\nabla p^*-\mathcal{F}(p^*)(u-p^*)-\frac{(\alpha(x,q^*)\nabla q^*)_{ss}}{2}\\\
    &\qquad \cdot(u-p^*)^2 \|_0 &&\\ 
    &\leq \|(\alpha(x,q^*)\nabla q^*)_{ss}\|_\infty\frac{\|u-p^*\|^2_0}{2}.
\end{flalign*}
Now we derive an upper bound for $G(u,u_{H,k}, p^*)$. Given the definition of  $  G(u,u_{H,k}, p^*)$ and Lipschitz continuity, we obtain
\begin{align*}
    G(u,u_{H,k}, p^*) &=\|(\alpha(x,p^*)\nabla v_{H,k} -\alpha(x,u)\nabla v_{H,k}+\alpha(x,u_{H,k})\nabla u_{H,k} -\alpha(x,p^*)\nabla u_{H,k})\|_0,\\
    & \leq \|\alpha(x,p^*-\alpha(x,u))\|_\infty \|\nabla v_{H,k}\|_0 +\|\alpha(x,p^*)-\alpha(x,u_{H,k})\|_\infty \|\nabla u_{H,k}\|_0\\
    &\leq (\Lambda_0 \|\nabla v_{H,k}\|_0\|u-p^*\|_\infty+\Lambda_0\|\nabla u_{H,k}\|_0\|u_{H,k}-p^*\|_\infty).
\end{align*}
 To estimate  $\|\nabla v_{H,k}\|_0$,  we use the definition of $v_{H,}$ and the stability of $Q^k$ and $I_H$ to obtain
 \[\|\nabla v_{H,k} \|_0=\|\nabla (\operatorname{id}-Q^k)I_Hu\|_0\lesssim  \|\nabla u\|_0. \qedhere\]
  \end{proof}

   \begin{remark} We discuss some implications of the above error estimates.
  \begin{itemize}
      
      \item We point out that the generic constants in the analysis above are independent of the mesh size and the spatial variations of $\alpha$ and mainly depend on the constants in the Assumptions~\ref{assum_1} and also on the constant ${C_\mathrm{ol}}$.
      \item From a theoretical point of view, no significant difference is expected between both linearization approaches, provided that the linearization point is chosen carefully to be close to both solutions. However, we observe a slight yet important difference. In the case of Fréchet derivative linearization, it is crucial to choose the linearization point carefully in such a way the conditions $\|\alpha_s(x,p^*)\nabla p^*\|_\infty \leq C_{p^*}$ and $C_IC_{p^*}H<1$ are satisfied, as seen in the numerical examples. %to ensure the convergence and ellipticity and consequently the well-posedness as we will discuss below. In contrast, there is no such condition for the case of Kačanov linearization. 
  \end{itemize}
\end{remark}

From a theoretical  as well as  practical point of view, the choice of the linearization point poses a challenge as it requires $p^*$ to be close to $u \text{ and } u_{H,k}$  despite neither $u $ nor $u_{H,k}$ being known a priori. Further,  the upper bound of the error estimates above seems to be suboptimal, e.g., choosing suitable norms when applying Hölder's inequality to estimate $G$.
\begin{itemize}
        \item Note that the upper bound in the error analysis depends on the infinity norm of both errors, namely $\|u-p^*\|_\infty$ and $\|u_{H,k}-p^*\|_\infty$. This suggests that the linearization point $p^*$ should be carefully selected close to $u$ and $u_{H,k}$ at the same time to ensure that $G(u,u_{H,k}, p^*)$ and $\eta(u)$ are sufficiently small. Then, we obtain the first order of convergence w.r.t. $H$, if the size of the patch satisfies $k \approx \operatorname{log}(H)$.
      \item To specify a good selection for the linearization points, observe that the term $\|\nabla u_{H,k}\|_0\|u_{H,k}-p^*\|_\infty$ can be evaluated a posteriori. We therefore mainly focus on suitable choices of $p^*$ that can make the value of $\|u-p^*\|_\infty$ very small. % \textcolor{red}{ We will further discuss the choice of the linearization in Section~\ref{SectionExperiment}.}
    \item  In order to appropriately choose the linearization point close to $u$, one might try to solve the problem using standard finite element method over coarse mesh.  Although $u_H$  does not provide an accurate  approximation, it can still reveal important qualitative and quantitative features of the weak solution. Through this numerical solution, one aim to gain insight into certain characteristics of $u$ such as its general behavior in which we obtain a valuable guidance for making an effective choice of $p^*$.
    \item A cascade of LOD solution might be a good choice to overcome the challenge of suitable linearization point selection as we also observed in the numerical experiments.  The point $p^*$ is updated iteratively such that  $\operatorname{ulod}$ solution is used in another iteration to compute a new multiscale space. However, this is not reflected in the current estimate as it would only accumulate the errors, namely all linearization errors arising initially from the choice of $p^*$ will add up. 
\end{itemize}

 \section*{Conclusion}
We presented and analyzed a numerical homogenization method for a class of nonmonotone quasilinear problems with rapidly varying coefficient. Local correction problems are linearized by two linearization techniques in order to solve linear problems and construct a problem-adapted basis of low-dimension multiscale space.  Numerical experiments illustrate the impact of the choice of the linearization points on the performance of the method. We proved   error estimates for both linearization techniques. Some linearization points perform better than others, depending on how close they are to the analytical solution. The two linearization techniques especially differ when the linearization point is far from the solution. For the latter situation, an iteration of our multiscale method by updating the linearization point improves the performance considerably. This indicates that iterative multiscale methods in the spirit of~\cite{MAir&barbara,IterativeLOD} may also be promising for nonmonotone quasilinear problems, which is left for future research.

\section*{Acknowledgments}
This work is funded by the Deutsche Forschungsgemeinschaft (DFG, German Research Foundation) under project number 496556642. BV also acknowledges support from the Deutsche Forschungsgemeinschaft (DFG, German Research Foundation) under Germany's Excellence Strategy – EXC-2047/1 – 390685813.

%\printbibliography
\bibliography{reference}
\bibliographystyle{plain}  
\end{document}